\documentclass{amsart}
\usepackage{amsmath,amssymb,amsthm}
\usepackage[backend=biber,
            style=alphabetic,
            url=false
            ]{biblatex}
\usepackage[utf8]{inputenc}

\usepackage{hyperref}

\usepackage[capitalize]{cleveref}
\usepackage{todonotes}
\usepackage{refcount}

\theoremstyle{plain}
\newtheorem{theorem}{Theorem}[section]
\newtheorem{proposition}[theorem]{Proposition}
\newtheorem{lemma}[theorem]{Lemma}
\newtheorem{cor}[theorem]{Corollary}
\newtheorem{claim}{Claim}[theorem]

\theoremstyle{definition}
\newtheorem{definition}{Definition}

\theoremstyle{remark}

\newcommand{\catspb}{CatSpec_{\approx}}
\newcommand{\A}{\mathcal{A}}
\newcommand{\B}{\mathcal{B}}
\newcommand{\ra}{\rightarrow}
\newcommand{\LR}{\Leftrightarrow}
\newcommand{\mrm}[1]{\mathrm{#1}}
\newcommand{\ol}[1]{\overline{#1}}
\renewcommand{\phi}{\varphi}
\DeclareMathOperator{\biemb}{\approx}

\begin{document}

\title{Degrees of bi-embeddable categoricity}

\author[N. Bazhenov]{Nikolay Bazhenov}
\author[E. Fokina]{Ekaterina Fokina}
\author[D. Rossegger]{Dino Rossegger}
\author[L. San Mauro]{Luca San Mauro}

\address{Sobolev Institute of Mathematics, 4 Acad. Koptyug Ave., Novosibirsk, 630090, Russia; and
Novosibirsk State University, 2 Pirogova St., Novosibirsk, 630090, Russia
}
\email{bazhenov@math.nsc.ru}

\address{Institute of Discrete Mathematics and Geometry, Vienna University of Technology,
Wiedner Hauptstra{\ss}e 8--10/104, 1040 Vienna, Austria
}

\email{ekaterina.fokina@tuwien.ac.at}
\email{dino.rossegger@tuwien.ac.at}
\email{luca.san.mauro@tuwien.ac.at}

\thanks{N. Bazhenov was supported by Russian Science Foundation, project No.~18-11-00028.
D.\ Rossegger was supported by RFBR, project no.~17-31-50026 mol\_nr.
E.\ Fokina was supported by the Austrian science fund FWF, project P~27527.  L.\ San Mauro was supported by the Austrian science fund FWF, projects P~27527 and M~2461.
}

\maketitle

\begin{abstract}
We investigate the complexity of embeddings between bi-embed\-dable structures. In analogy with categoricity spectra, we define the bi-embeddable categoricity spectrum of a structure $\A$ as the family of Turing degrees that compute embeddings between any computable bi-embeddable copies of $\A$; the degree of bi-embeddable categoricity of $\A$ is the least degree in this spectrum (if it exists). We extend many known results about categoricity spectra to the case of bi-embeddability. In particular, we exhibit structures without degree of bi-embeddable categoricity, and we show that every degree d.c.e\ above $\mathbf{0}^{(\alpha)}$ for $\alpha$ a computable successor ordinal and $\mathbf{0}^{(\lambda)}$ for $\lambda$ a computable limit ordinal is a degree of bi-embeddable categoricity. We also give examples of families of degrees that are not bi-embeddable categoricity spectra.
\end{abstract}
\section{Introduction}
Two mathematical structures are considered the same if they are isomorphic. While this classification is valid for structural properties of structures, for computational properties it is too coarse. Indeed, two isomorphic structures can have very different computational properties. Even two  isomorphic computable structures may have different algorithmic properties. Fr\"ohlich and Shepherdson~\cite{frohlich1956}, and independently, Malt'sev~\cite{maltsev1962} discovered that there are isomorphic computable structures which behave differently computationally in the sense that in one structure an additional relation is computable while in the other one it is not. They concluded that there can not be a computable isomorphism between these structures, since any two computably isomorphic structures must have the same algorithmic properties. Since this discovery, the study of the complexity of isomorphisms between computable structures has been one of the main themes of effective mathematics and computable structure theory in particular. One of the main goals in the area is to find connections between the structural properties of structures, as defined by their isomorphism types, and the computational properties its isomorphic copies might possess.

In this article we extend this study to investigate the complexity of embeddings between bi-embeddable structures. Two structures $\A$ and $\B$ are \emph{bi-embeddable} (notation: $\A\approx\B$) if there is an embedding of either in the other. The bi-embeddability relation has attracted a lot of attention of specialists in computable structure theory and descriptive set theory in recent years (see, e.g, \cite{louveau2005,friedman2011}). Montalb\'an~\cite{montalban2005} showed that every hyperarithmetic linear ordering is bi-embeddable with a computable one, and in~\cite{greenberg2008}, together with Greenberg, they showed that the same is true for abelian $p$-groups, Boolean algebras, and compact metric spaces. In~\cite{fokina2018a}, Fokina, Rossegger, and San Mauro, observed that every equivalence structure is bi-embeddable with a computable one. These results show that in many natural classes of structures, one of the main notions one usually uses to measure the complexity of a structure, its degree spectrum\footnote{The degree spectrum of a structure is the family of Turing degrees of its isomorphic copies. This notion is easily generalized to work with bi-embeddability by considering the family of Turing degrees of the bi-embeddable copies of a structure.}, fails to properly capture the desired computational content. This motivates the systematic study of the complexity of embeddings between bi-embeddable structures.
We develop this in analogy to the study of the complexity of isomorphisms between computable structures. Our main notion is the following.
%
\begin{definition}\label{def:001}
	Let $\mathbf{d}$ be a Turing degree. We say that a computable structure $\mathcal{S}$ is \emph{$\mathbf{d}$-computably bi-embeddably categorical} (or \emph{$\mathbf{d}$-computably b.e.\ categorical}, for short) if for any computable structure $\mathcal{A}\approx \mathcal{S}$, there are $\mathbf{d}$-computable isomorphic embeddings $f\colon \mathcal{A} \hookrightarrow \mathcal{S}$ and $g\colon \mathcal{S} \hookrightarrow \mathcal{A}$. The \emph{bi-embeddable categoricity spectrum} of $\mathcal{S}$ is the set
		\[
		\catspb(\mathcal{S}) = \{ \mathbf{d}\,\colon \mathcal{S} \text{ is } \mathbf{d}\text{-computably bi-embeddably categorical}\}.
		\]
	A degree $\mathbf{c}$ is the \emph{degree of bi-embeddable categoricity} of $\mathcal{S}$ if $\mathbf{c}$ is the least degree in the spectrum $\catspb(\mathcal{S})$.
\end{definition}
Notice that the bi-embeddable categoricity spectrum of a structure does not necessarily have a degree of bi-embeddable categoricity. We study such structures in \cref{sec:nodeg}.

Definition~\ref{def:001} is similar to the notions of categoricity spectrum and degree of categoricity which were introduced in~\cite{fokina2010}. The \emph{categoricity spectrum} of a computable structure $\mathcal{S}$ is the set of all Turing degrees which are capable of computing isomorphisms among arbitrary computable isomorphic copies of $\mathcal{S}$. The \emph{degree of categoricity} of $\mathcal{S}$ is the least degree from the categoricity spectrum of $\mathcal{S}$. In recent years, researchers have been extensively investigated what classes of Turing degrees can be categoricity spectra  \cite{fokina2010,csima2013,MS-15,bazhenov2017c} and what can not \cite{anderson2016,FS-14,FT-18}. In the present paper we will discuss in detail to which extent such results can be transferred to the realm of bi-embeddability.

\cref{def:001} already appeared in our article~\cite{bazhenov2018a}, where 	 we gave a complete characterization of the degrees of bi-embeddable categoricity of equivalence structures by showing that these are either $\mathbf{0}$, $\mathbf{0}'$, or $\mathbf{0}''$. In this article we focus on general results, especially, on the question which Turing degrees can and can not be degrees of categoricity. Some of the results of the paper were announced in~\cite{bazhenov2018b}.
%
\section{Examples of bi-embeddable categoricity spectra}
We now give several examples of bi-embeddable categoricity spectra. In particular we exhibit structures without degree of bi-embeddable categoricity and show that every degree d.c.e\ above $\mathbf{0}^{(\alpha)}$ for $\alpha$ a computable successor ordinal and $\mathbf{0}^{(\lambda)}$ for $\lambda$ a computable limit ordinal is a degree of bi-embeddable categoricity. Our examples have in common that they are only bi-embeddable with their isomorphic copies. We call such structures b.e.\ trivial. More formally, a structure $\mathcal S$ is \emph{b.e.\ trivial} if
\[ Iso(\mathcal S):= \{ \mathcal A : \mathcal A\cong \mathcal S\}=\{ \mathcal A: \mathcal A\biemb \mathcal S\}=:BiEmb(\mathcal S).\]
B.e.\ triviality has been thoroughly studied in the context of degree spectra in~\cite{fokina2018a}. For b.e.\ trivial structures there is a strong connection between computable categoricity and computable bi-embeddable categoricity. In particular, if a b.e.\ trivial structure is $\mathbf d$-computably categorical for some degree $\mathbf{d}$, then it is $\mathbf d$-computably bi-embeddably categorical.

The main results in this section are stated in \cref{theo:degrees,theo:limitdegree}. Their proofs follow the ideas of the proofs of similar theorems for degrees of categoricity given in~\cite{csima2013}. The key ingredient of the proofs is jump inversion using pairs of structures. Csima, Franklin, and Shore used back-and-forth trees to accomplish the jump inversion. These trees have the downside that they are not b.e.\ trivial. We therefore jump invert using pairs of well-orderings. This technique has recently been used in~\cite{CFG+,bazhenov2017c,rossegger2018}. In what follows we fix a path $\mathcal P$ through Kleene's $\mathcal O$ and identify computable ordinals with their notations on this path. We will not distinguish between the notation of an ordinal on $\mathcal P$ and the ordinal itself. However, there should not arise any confusion as what we mean should be clear from the context.
\begin{theorem}\label{theo:degrees}
	Let $\alpha$ be a computable successor ordinal. Suppose that $\mathbf{d}$ is a Turing degree such that $\mathbf{d}$ is d.c.e.\ in $\mathbf{0}^{(\alpha)}$ and $\mathbf{d} \geq \mathbf{0}^{(\alpha)}$. There is a computable, bi-em\-bed\-da\-bly trivial structure $\mathcal{S}$ with degree of bi-embeddable categoricity $\mathbf{d}$.
\end{theorem}
\begin{theorem}\label{theo:limitdegree}
	Let $\alpha$ be a computable limit ordinal. There is a computable, bi-embeddably trivial structure $\mathcal{S}$ with degree of bi-embeddable categoricity $\mathbf 0^{(\alpha)}$.
\end{theorem}
Before we give the proofs of the above theorems we need to recall some preliminaries. We assume that the reader is familiar with computable infinitary logic. If they are not, we suggest Ash and Knight~\cite{ash2000} as reference. Recall that a family of computable infinitary formulas $\Psi$ is a \emph{formally $\Sigma^0_\alpha$ Scott family} for a structure $\A$ with parameters $\ol c$ if
\begin{enumerate}
  \item $\Psi$ is c.e.,
  \item every formula in $\Psi$ is computable $\Sigma_\alpha$,
  \item for every $\ol a\in A^{<\omega}$ there is a unique formula $\phi_{\ol a}\in \Psi$ such that $(\A,\ol a)\models \phi_{\ol a}(\ol a,\ol c)$,
  \item and for $\ol a,\ol b\in A^{<\omega}$ of the same length if $\phi_{\ol a}=\phi_{\ol b}$, then there is an automorphism of $\A$ taking $\ol a$ to $\ol b$.
\end{enumerate}
In other words, $\Psi$ is a c.e.\ family of computable $\Sigma_\alpha$ formulas defining the automorphism orbits of $\A$. It follows from a classical result by Ash, Knight, Manasse, and Slaman~\cite{ash1989} that structures having formally c.e. $\Sigma^0_\alpha$ Scott families are $\Delta^0_\alpha$ categorical. See~\cite{ash2000} for more background on this topic.

In order to prove \cref{theo:degrees,theo:limitdegree} we still need to establish some properties of the well-orderings we will use for jump inversion. In the case where $\alpha=2\beta+1$ we will use the ordinals $\omega^\beta$ and $\omega^\beta\cdot 2$ and in the case where $\alpha=2\beta+2$ we will use $\omega^{\beta+1}$ and $\omega^{\beta+1}+\omega^\beta$. For limit ordinals we will use the corresponding successor ordinals obtained from their fundamental sequences. The following will be central to our proofs.
\begin{lemma}[{\cite[Proposition 2]{bazhenov2017c}}]\label{lem:switching}
  Assume that $\alpha$ is a non-zero computable ordinal, and $n$ is a natural number. Suppose that $\mathcal M$ and $\mathcal N$ are computable structures, $\Psi$ is a formally $\Sigma^0_\alpha$ Scott family for $\mathcal M$ without parameters, $\Xi$ is a formally $\Sigma^0_\alpha$ Scott family for $\mathcal N$ with parameters $\ol c$. Assume that $\Psi\subset \Xi$ and that $\xi(\ol x)$ is a $\Sigma^c_{\alpha+n}$ formula such that $\mathcal M\models \neg \exists \ol x \xi(\ol x)$ and $\ol c$ is the unique tuple from $\mathcal N$ satisfying $\xi(\ol x)$. Then, given computable indices of computable structures $\mathcal A$ and $\mathcal B$ such that $\mathcal A\cong \mathcal B\cong \mathcal C\in \{\mathcal M,\mathcal N\}$, one can effectively determine a $\Delta^0_{\alpha+n}$ index for an isomorphism $F$ from $\mathcal A$ onto $\mathcal B$.
\end{lemma}
We will use the following relations on linear orderings.
\begin{definition} Let $\mathcal L$ be a linear ordering and $x,y\in L$. Then let
  \begin{enumerate}
    \item $x\sim_0 y$ if $x=y$,
    \item $x \sim_{1} y$ if $[x,y]$ or $[y,x]$ is finite,
    \item for $\alpha=\beta+1$, $x \sim_{\alpha} y$ if in $\mathcal L /{\sim_{\beta}}$, $[x]_{\sim_\beta} \sim_1 [y]_{\sim_\beta}$,
    \item for $\alpha=\lim \beta$, $x \sim_{\alpha} y $ if for some  $\beta < \alpha$, $x \sim_{\beta} y$.
  \end{enumerate}
  \noindent The relation $\sim_1$ is commonly known as the block relation. We refer to $\sim_\alpha$ as the $\alpha$-block relation.
\end{definition}
The $\alpha$-block relation is relatively intrinsically $\Sigma^0_{2\alpha}$. To see this first note that for $\alpha=\beta+1$, $\sim_{\alpha}$ is definable by
\[ x\sim_\alpha y \LR \bigvee_{n\in \omega} \forall y_1,\dots, y_n \left(x<y_1<\cdots<y_n<y \ra  \bigvee_{1\leq  i<j\leq n } y_i \sim_\beta y_j \right).\]
For $\lambda$ a limit ordinal, let $\phi$ be a fundamental sequence of $\lambda$ in $\mathcal P$. Then $\sim_\lambda$ is definable by
\[ x\sim_\lambda y \LR \bigvee_{i\in\omega} x\sim_{\phi(i)} y.\]
Using transfinite induction it is immediate from the definition that for each computable ordinal $\alpha$, $\sim_\alpha$ is definable by a computable $\Sigma_{2\alpha}$ formula and thus relatively intrinsically $\Sigma^0_{2\alpha}$.
We are now ready to show that our pairs $\omega^{\beta}$, $\omega^{\beta}\cdot 2$ and $\omega^{\beta+1}$, $\omega^{\beta+1}+\omega^{\beta}$ satisfy the conditions of \cref{lem:switching}. The lemmas follow from the proofs in~\cite{ash1986} (see also~\cite[Theorem 17.5]{ash2000}). We sketch the proofs for the sake of completeness.
\begin{lemma}\label{lem:oddordinalsscottfam}
  The ordering $\omega^\beta$ has a formally $\Sigma^0_{2\beta}$ Scott family $\Psi$ without parameters, and $\omega^\beta\cdot 2$ has a formally $\Sigma^0_{2\beta}$ Scott family with one parameter $c$ such that $\Psi\subseteq \Xi$ and there is a $\Sigma_{2\beta+1}$ formula $\xi(x)$ such that $\omega^{\beta}\cdot 2\models \xi(c)$ but no element of $\omega^{\beta}$ satisfies $\xi(x)$.
\end{lemma}
\begin{proof}
  Since well-orderings are rigid it is sufficient to give a defining family, i.e., a family of formulas such that each element satisfies a formula in the family and not two formulas are satisfied by two elements. From this it is easy to obtain the Scott family of the ordering. We therefore give $\Psi$ and $\Xi$ as defining families instead. Towards this notice that for every non-zero ordinal $\gamma<\omega^\beta$ there is a computable $\Sigma_{2\beta}$ formula $\theta_\gamma(x)$ without parameters such that for any well-ordering $\mathcal W$
  \[  (\mathcal W, a) \models \theta_\gamma(a) \text{ iff } [0^{\mathcal W},a)\cong \gamma, \]
  where $0^{\mathcal W}$ is the first element of $\mathcal W$~\cite[Proposition 7.2]{ash2000}. Let $\Psi=\{ \theta_\gamma(x): \gamma<\omega^{\beta} \}$. Then it is not hard to see that this is a defining family for $\omega^{\beta}$. Let $c$ be the first element of the second copy of $\omega^{\beta}$ in $\omega^{\beta}\cdot 2$. Then $\Xi$ consists of all the formulas of $\Psi$ and formulas $\theta_\gamma(x,y)$ such that
  \[ (\mathcal W,c,a)\models \theta_\gamma(c,a) \text{ iff } [c,a)\cong \gamma.\]
  Clearly $\Xi$ is a $\Sigma_{2\beta}$ defining family for $\omega^{\beta}\cdot 2$ with one parameter. Furthermore the parameter $c$ is definable in $\omega^{\beta}\cdot 2$ by the $\Sigma_{2\beta+1}$ formula
  \[ \xi(x)=\exists y\leq x \land\forall z<x \ z\not \sim_\beta x.\]
\end{proof}
\begin{lemma}
  The ordering $\omega^{\beta+1}$ has a formally $\Sigma^0_{2\beta+2}$ Scott family $\Psi$ without parameters, and $\omega^{\beta+1}+\omega^{\beta}$ has a formally $\Sigma^0_{2\beta+2}$ Scott family with one parameter $c$ such that $\Psi\subseteq \Xi$ and there is a $\Pi_{2\beta+1}$ formula $\xi(x)$ such that $\omega^{\beta+1}+\omega^{\beta}\models \xi(c)$ but no element of $\omega^{\beta}$ satisfies $\xi(x)$.
\end{lemma}
\begin{proof}
  The construction of the Scott families $\Psi$ and $\Xi$ is analogous to the construction of $\Psi$ and $\Xi$ in \cref{lem:oddordinalsscottfam}. The parameter $c$ of $\Xi$ is the first element of the last copy of $\omega^\beta$ in $\omega^{\beta+1}+\omega^{\beta}$. It is definable by the $\Pi_{2\beta+1}$ formula
\[
	\xi(x) = \forall z \geq x ( z\sim_{\beta} x) \land \forall z < x (z \not\sim_{\beta} x).
\]

\end{proof}
We are now ready to prove \cref{theo:degrees}.
\begin{proof}[Proof of \cref{theo:degrees}]

We build two b.e.\ trivial computable structures $\A$ and $\B$ such that $\mathcal{A}\cong \mathcal{B}$, $\mathcal{A}$ is $\mathbf{d}$-computably categorical, and any embedding from $\mathcal{A}$ into $\mathcal{B}$ must compute $\mathbf{d}$. We first give the construction for the case when $\mathbf{d}$ is d.c.e.\ over $\mathbf{0}^{(2\beta+1)}$, where $\beta$ is an infinite ordinal. For finite ordinals the construction is the same except for a shift of indices by $1$.

Ash's characterization of the back-and-forth relations for linear orders and his pairs of structures theorem, see Chapters 11 and 16 in~\cite{ash2000}, tells us that for any $\Sigma^0_{2\beta+1}$ set $S$, there is a computable sequence $(C_e)_{e\in \omega}$ of linear orders such that
\begin{equation} \label{equ:01}
C_e\cong\begin{cases}
\omega^{\beta}\cdot 2 & \text{if~} e\in S,\\
\omega^\beta & \text{if~} e\not \in S.
\end{cases}
\end{equation}

A relativized version of the argument from~\cite[Theorem 3.1]{fokina2010} shows that one can choose a set $D\in\mathbf{d}$ such that $D$ is d.c.e. in $\mathbf{0}^{(2\beta+1)}$ and for any oracle $X$, we have:
\[
	(\overline{D} \text{~is c.e. in~} X)\ \Rightarrow\ D\leq_T X\oplus \mathbf{0}^{(2\beta+1)}.
\]

As $D$ is d.c.e.\ above $\mathbf{0}^{(2\beta+1)}$ we have that $D=U\setminus V$ for $U$ and $V$ c.e.\ in $\mathbf{0}^{(2\beta+1)}$ and we may assume that $V\subset U$.
The language of our structures contains an equivalence relation $\sim$, a partial order $\leq$, a unary predicate $T$, and a unary predicate $P_e$, for each $e\in \omega$.  We first describe the construction of $\A$. For every $e$, we choose elements $a_e$ and $b_e$ in $\A$, and for every $P_e$, we let $P_e(A)$ be infinite and include $a_e$, $b_e$.

For a fixed $e$, we give the construction for the substructure on $P_e(A)$. We let $P_e(A)$ consist of two infinite equivalence classes (with respect to $\sim$) such that $a_e\not \sim b_e$.
The two classes $[a_e]$ and $[b_e]$ will both contain pairs of linear orders, i.e., structures of the form $(L_1,L_2)$ where $L_1$ and $L_2$ are linear orders (with respect to $\leq$), any $x\in L_1$ and $y\in L_2$ are incomparable, and $T([a_e])=L_1$.

If $e=2m$, then we encode the information whether or not $m$ is an element of $D$ in $P_e(A)$.  There are three cases:
\begin{enumerate}
  \item $m\not \in U$: we build $T([a_e]),\neg T([a_e]),T([b_e])\cong \omega^{\beta}$, and $\neg T([b_e])\cong \omega^{\beta}\cdot 2$;
  \item $m\in U\setminus V$: we build $T([b_e])\cong \omega^\beta$ and $T([a_e]),\neg T([a_e]),\neg T([b_e])\cong \omega^{\beta}\cdot 2$;
  \item $m\in V$: we build $T([a_e]),T([b_e]),\neg T([a_e]),\neg T([b_e])\cong \omega^{\beta}\cdot 2$.
\end{enumerate}
Analyzing this construction, we see that
\[ [a_e]\cong \begin{cases}
(\omega^{\beta}\cdot 2, \omega^{\beta}\cdot 2) &\text{if~} m\in U,\\
  (\omega^\beta,\omega^\beta) & \text{if~} m\not \in U,
\end{cases}
\quad \text{and}\quad
[b_e]\cong \begin{cases}
(\omega^{\beta}\cdot 2,\omega^{\beta}\cdot 2)& \text{if~} m\in V, \\
 (\omega^{\beta},\omega^{\beta}\cdot 2) & \text{if~} m\not \in V.
\end{cases}
\]

If $e=2m+1$, then we let $[b_e]\cong (\omega^{\beta},\omega^{\beta}\cdot 2)$, and for $[a_e]$ we let
\[ [a_e]\cong \begin{cases}
(\omega^{\beta}\cdot 2, \omega^{\beta}\cdot 2) & \text{if~} m\in \emptyset^{(2\beta+1)},\\
(\omega^{\beta},\omega^{\beta}) & \text{if~} m\not \in \emptyset^{(2\beta+1)}.
\end{cases}\]
The existence of the uniformly computable sequence of structures $(C_e)_{e\in \omega}$ from~(\ref{equ:01}) implies that we can do the construction computably.

For $\B$, we again choose elements $\hat a_e$, $\hat b_e$ for every $e$, and for $e=2m$ we build $\mathcal{B}$ like $\A$ with the difference that the roles of $\hat a_e$ and $\hat b_e$ are switched. For $e=2m+1$ we let
\[
	[\hat a_e]=\begin{cases}
	(\omega^\beta \cdot 2, \omega^\beta\cdot 2) & \text{if~} m\in \emptyset^{(2\beta+1)},\\
	(\omega^\beta,\omega^\beta\cdot 2) &\text{if~} m\not\in \emptyset^{(2\beta+1)},
	\end{cases}
	\quad \text{and}\quad
	[\hat b_e]=\begin{cases}
	(\omega^\beta,\omega^\beta\cdot 2) & m\in\emptyset^{(2\beta+1)},\\
	(\omega^\beta,\omega^\beta) & m\not\in\emptyset^{(2\beta+1)}.
	\end{cases}
\]
Clearly, $\B$ and $\A$ are isomorphic and computable. It is not hard to show that they are b.e.\ trivial: Indeed, every embedding of $\A$ into a bi-embeddable copy $\hat \A$ must map elements in $P_e(A)$ to elements in $P_e(\hat A)$, for every $e\in \omega$. Every $P_e(\hat A)$ must have exactly two equivalence classes as otherwise $P_e(\hat A)\not \approx P_e(A)$. Moreover, the pairs of structures that we use are pairs of well-orders, and thus b.e.\ trivial.

\begin{claim}
	The structure $\mathcal A$ is $\mathbf d$-computably categorical.
\end{claim}
\begin{proof}
	Let $\mathcal B$ and $\mathcal C$ be computable copies of $\mathcal A$. Clearly every isomorphism must map $P_e(B)$ to $P_e(C)$. Fix some $e\in \omega$.
	We produce a $\mathbf d$-computably partial isomorphism from $P_e(B)$ to $P_e(C)$. That there is a $\mathbf d$-computable isomorphism $\mathcal B\rightarrow\mathcal C$ will follow from the fact that our construction does not depend on the choice of $e$.

	Note that the formula $\xi(x)$ given in \cref{lem:oddordinalsscottfam} is computably $\Sigma_{2\beta+1}$ and that we can restrict the quantifiers in this formula to elements in $P_e(B) \cap T(B)$ (respectively $P_e(B)\cap \neg T(B)$), and to elements that are in the same $\sim$ equivalence class as $x$,  without changing its complexity. Let $\tilde \xi_e^T(x)$ and $\tilde \xi_e^{\neg T}(x)$ be the relativized formula. These formulas are also $\Sigma_{2\beta+1}$. It is thus c.e.\ in $\mathbf d$ to find an element $x\in B$ or $C$ of which these formulas hold.

	We distinguish the following cases.
	\begin{enumerate}
		\item $e=2m+1$ and $m\in \emptyset^{(2\beta+1)}$. Search for elements $b\in B$ and $c\in C$ that satisfy $\tilde\xi_e^T(x)$. By construction we will find such elements and $[b]\cong[c]\cong (\omega^\beta\cdot 2,\omega^\beta\cdot 2)$. Using \cref{lem:oddordinalsscottfam} and \cref{lem:switching} $\mathbf d$ can compute a partial isomorphism between $[b]$ and $[c]$. Now look for two elements $\hat b\in P_e(B)$ and $\hat c\in P_e(C)$ such that $\hat b\not \sim b$ and $\hat c\not \sim c$. Then $[\hat b]\cong [\hat c]\cong (\omega^{\beta},\omega^{\beta}\cdot 2)$ and again by \cref{lem:oddordinalsscottfam} and \cref{lem:switching} $\mathbf d$ can compute a partial isomorphism between $[\hat b]$ and $[\hat c]$. We thus get a partial isomorphism from $P_e(B)$ to $P_e(C)$.
		\item $e=2m+1$ and $m\not\in \emptyset^{(2\beta+1)}$. This case is similar to (1).
		\item $e=2m$ and $m\in D$. Then $m\in U\setminus V$ and the construction proceeds similarly to the two former cases.
		\item $e=2m$ and $m\not \in D$. Search for elements $b\in B$ and $c\in C$ satisfying $\tilde\xi_e^{\neg T}(x)$. By construction we will find such elements, get that $[b]\cong [c]$, and obtain that $\mathbf d$ can compute a partial isomorphism between $[b]$ and $[c]$. Then find elements $\hat b\in P_e(B)$ and $\hat c\in P_e(C)$ such that $b\not \sim \hat b$ and $c\not \sim \hat c$. These elements again exist by construction and from \cref{lem:switching} and \cref{lem:oddordinalsscottfam} we obtain a $\mathbf d$-computable partial isomorphism from $[\hat b]$ to $[\hat c]$.
	\end{enumerate}
\end{proof}
It remains to show that every embedding $f\colon\A \hookrightarrow \B$ computes $D$. We have that $f\geq_T \mathbf{0}^{(2\beta+1)}$ because
\[ m\in \emptyset^{(2\beta+1)}\LR f(a_{2m+1})\sim \hat b_{2m+1} \quad \text{and} \quad m\not\in\emptyset^{(2\beta+1)} \LR f(a_{2m+1})\sim \hat a_{2m+1}. \]
Similarly, we have that
\[  m \not\in U\setminus V \LR (f(a_{2m})\sim \hat a_{2m}) \text{~or~} (m\in V).
\]
Thus, $\overline{D}$ is c.e. in $f\oplus \mathbf{0}^{(2\beta+1)}$. Hence, $D \leq_T (f\oplus \mathbf{0}^{(2\beta+1)}) \equiv_T f$.

The construction for the case $\alpha=2\beta+2$ is nearly the same. The only difference is that in place of~(\ref{equ:01}), we use the following fact: For any $\Sigma^0_{2\beta+2}$ set $S$, there is a computable sequence $(C_e)_{e\in \omega}$ of linear orders such that
\[
C_e\cong\begin{cases}
\omega^{\beta+1}+\omega^{\beta} & \text{if~} e\in S,\\
\omega^{\beta+1} & \text{if~} e\not \in S.
\end{cases}
\]

\end{proof}
The main ideas behind the proof of \cref{theo:limitdegree} are similar to those used in the successor case. But we have to take into account that if $\alpha$ is a limit ordinal, then the definition of $\emptyset^{(\alpha)}$ is different from the successor case. We will use that $\emptyset^{(\alpha)}\equiv_T \{ \langle i,n\rangle\in\omega:n\in \emptyset^{(\phi(i))}\}$ where $\phi$ is a fundamental sequence for $\alpha$. To do this we will use pairs $(\omega^{\beta}\cdot 2, \omega^{\beta})$ not only for fixed $\beta$ but for infinitely many different $\beta$ below $\alpha$. Ash and Knight~\cite{ash1990a}, see also~\cite[Theorem 18.9]{ash2000}, proved a variation of the pairs of structure theorem which works for our purposes. We state it here in slightly different terminology.
\begin{lemma}\label{lem:uniformpairsofstructures}
  Let $(e_i, \alpha_i)_{i\in\omega}$ be a computable sequence of $\Pi^0_{\alpha_i}$ sets $P_{e_i}$ and let $(\A_i,\mathcal B_i)_{i\in\omega}$ be a sequence of structures such that $\mathcal B_i\leq_{\alpha_i} \A_i$ and $\{\A_i,\mathcal B_i\}$ is $\alpha_i$-friendly, uniformly in $i$. Then there is a uniformly computable sequence of structures $(\mathcal C_{\langle i,n\rangle})_{\langle i,n\rangle\in \omega}$ such that
  \[
    \mathcal C_{\langle i,n\rangle}\cong \begin{cases}
    \A_i & \text{if~} n\in P_{e_i},\\
    \mathcal{B}_i&\text{otherwise}.
    \end{cases}
  \]\end{lemma}
\begin{proof}[Proof of \cref{theo:limitdegree}]
  As in the proof of \cref{theo:degrees} we will build two b.e.\ trivial computable structures $\A\cong\mathcal B$, such that $\mathcal A$ is $\mathbf 0^{(\alpha)}$-computably categorical and any embedding from $\mathcal A$ into $\mathcal B$ must compute $\mathbf 0^{(\alpha)}$. Let $\phi$ be a fundamental sequence for $\alpha$ such that without loss of generality for all $i\in \omega$, $\phi(i)\cong 2\beta+1$ for some $\beta<\alpha$. By definition we have that
  \[\emptyset^{(\alpha)}\equiv_T \{ \langle i,n\rangle\in\omega: n\in \emptyset^{(\phi(i))}\}=\emptyset^{\phi}.\]

  We can now use \cref{lem:uniformpairsofstructures} with $P_{e_i}=\{ x: \langle i,x\rangle \in \emptyset^\phi\}=\emptyset^{(\phi(i))}$,
  and $(\A_i,\mathcal B_i)=(\omega^{\beta_i}\cdot 2, \omega^{\beta_i})$ where $\beta_i$ is such that $\phi(i)=2\beta_i+1$.
  Our structures $\A$ and $\B$ are similar to the successor cases with the exception that for our designated elements $a_e$, $b_e$ and $\hat a_e$, $\hat b_e$ where $e=\langle i,n\rangle$ we let
  \[ [a_e]\cong \begin{cases}
  (\omega^{\beta_i}\cdot 2, \omega^{\beta_i}\cdot 2) & \text{if~} n\in \emptyset^{(\phi(i))},\\
  (\omega^{\beta_i},\omega^{\beta_i}) & \text{if~} n\not\in \emptyset^{(\phi(i))},
  \end{cases}
  \quad \text{and}\quad
  [b_e]\cong (\omega^{\beta_i},\omega^{\beta_i}\cdot2),
  \]
  and we let
  \[ [\hat a_e]\cong \begin{cases}
  (\omega^{\beta_i}\cdot 2, \omega^{\beta_i}\cdot 2) & \text{if~} n\in \emptyset^{(\phi(i))},\\
  (\omega^{\beta_i},\omega^{\beta_i}\cdot 2) & \text{if~} n\not\in \emptyset^{(\phi(i))},
  \end{cases}
  \quad \text{and}\quad
  [\hat b_e]\cong \begin{cases}
  (\omega^{\beta_i}, \omega^{\beta_i}\cdot 2) & \text{if~} n\in \emptyset^{(\phi(i))},\\
  (\omega^{\beta_i},\omega^{\beta_i}) & \text{if~} n\not\in \emptyset^{(\phi(i))}.
  \end{cases}
  \]
  \cref{lem:uniformpairsofstructures} implies that $\mathcal A$ and $\mathcal B$ are computable.
  That $\A$ is $\mathbf 0^{(\alpha)}$-computably categorical follows by a similar argument as in the successor case. First notice that elements satisfying $P_e$ must be sent to $P_e$. Thus, we may fix $e=\langle i,n\rangle$. By the same arguments as in the proof of \cref{theo:degrees}, $\mathbf{0}^{(2\beta_i+1)}$ can compute a partial isomorphism between the $P_e$ substructures of arbitrary computable copies of $\A$. As the $\beta_i$ are bounded by $\alpha$ we have that $\mathbf{0}^{(\alpha)}$ can compute a partial isomorphism between all substructures on $P_e$, $e\in\omega$, uniformly in $e$. It follows that $\mathbf 0^{(\alpha)}$ can compute an isomorphism.

  It remains to prove that every embedding between $\A$ and $\mathcal B$ computes $\mathbf{0}^{(\alpha)}$. It is sufficient to show that every embedding $f:\A\hookrightarrow \mathcal B$ computes $\emptyset^\phi$. This is the case as
  \[ m=\langle i,n\rangle\in \emptyset^\phi \LR n\in\emptyset^{(\phi(i))}\LR f(a_m)=\hat a_m\]
  and likewise $m\not \in \emptyset^\phi\LR f(a_m)=\hat b_m$. As $\deg_T(\emptyset^\phi)= \mathbf{0}^{(\alpha)}$, this proves the theorem.
\end{proof}



\subsection{Structures with no degree of b.e. categoricity}
\label{sec:nodeg}
Here we build examples of b.e.~categoricity spectra with no least degree.
The exposition mainly follows~\cite{MS-15,Bazh-19}.

In this section, trees are treated as substructures of $\omega^{<\omega}$. For a tree $T$, the \emph{branching function} $b_T\colon T\to \omega \cup \{\omega\}$ gives the number of children of a node $\sigma$ from $T$, or more formally:
\[
	b_T(\sigma) = card(\{ n\in\omega\,\colon \sigma\widehat{\ }\langle n\rangle \in T\}).
\]

Let $X\subseteq \omega$ be an oracle. A Turing degree $\mathbf{d}$ is a \emph{PA degree over} $X$ if for any infinite $X$-computable, finite-branching tree $T$ with an $X$-computable branching function $b_T$, there is a $\mathbf{d}$-computable (infinite) path through $T$. Note that the notion of a $PA$ degree over $X$ depends only on the choice of the Turing degree of a set $X$.

The main result of the section is the following

\begin{theorem} \label{theo:PA-deg}
	Let $\alpha$ be a computable non-limit ordinal. Then there is a b.e. trivial computable structure $\mathcal{M}$ such that the b.e. categoricity spectrum of $\mathcal{M}$ is equal to the set of $PA$ degrees over $\mathbf{0}^{(\alpha)}$.
\end{theorem}

Before the proof of the theorem, we recall some known facts about $PA$ degrees:
\begin{itemize}
	\item[(a)] For any $X$, the set of $PA$ degrees over $X$ is upwards closed (see, e.g., Theorem 6.2 in~\cite{Sim77}).

	\item[(b)] If $\mathbf{d}$ is a $PA$ degree over $X$, then there is a degree $\mathbf{c}$ such that $\mathbf{c} < \mathbf{d}$ and $\mathbf{c}$ is also a $PA$ degree over $X$ (Theorem 6.5.i in~\cite{Sim77}). In other words, the set of $PA$ degrees over $X$ does not have minimal elements.

	\item[(c)] A degree $\mathbf{d}$ is a PA degree over $\mathbf{0}$ if and only if $\mathbf{d}$ computes a complete extension of Peano arithmetic.
\end{itemize}

Our proof of Theorem~\ref{theo:PA-deg} heavily uses the following characterization of $PA$ degrees, obtained by Scott~\cite{Scott62}, Jockusch and Soare~\cite{JS72}, and Solovay:

\begin{proposition}[see Theorem~6.6 in~\cite{Sim77}]\label{prop:PA-char}
	A Turing degree $\mathbf{d}$ is a $PA$ degree over $X$ if and only if there is a $\mathbf{d}$-computable set $A$ with the following properties:
	\[
		\{ e\,\colon \varphi^X_e(e)\downarrow\ = 1\} \subseteq A  \text{ and }
		\{ e\,\colon \varphi^X_e(e)\downarrow\ = 0\} \subseteq \overline{A}.
	\]
\end{proposition}

\

\begin{proof}[Proof of \cref{theo:PA-deg}]

Here we give a detailed proof for the case when $\alpha = 2\beta+2$ and $\beta \geq \omega$.

Note that the set $\{ e \,\colon \varphi^{\emptyset^{(2\beta + 2)}}_e(e)\downarrow\  = 1 \}$ is $\Sigma^0_{2\beta+2}$. Therefore, as in Theorem~\ref{theo:degrees}, the results of Ash and Knight allow us to build two computable sequences of linear orders $(\mathcal{A}_{e})_{e\in\omega}$ and $(\mathcal{B}_{e})_{e\in\omega}$ such that:
\begin{gather*}
	\mathcal{A}_{e} \cong
	\begin{cases}
		\omega^{\beta+1} + \omega^{\beta} & \text{if } \varphi^{\emptyset^{(2\beta + 2)}}_e(e)\downarrow\  = 1,\\
		\omega^{\beta+1} & \text{otherwise};
	\end{cases}
	\\
	\mathcal{B}_e \cong
	\begin{cases}
		\omega^{\beta+1} + \omega^{\beta} & \text{if } \varphi^{\emptyset^{(2\beta + 2)}}_e(e)\downarrow\  = 0,\\
		\omega^{\beta+1} & \text{otherwise}.
	\end{cases}
\end{gather*}

The desired computable structure $\mathcal{M}$ is arranged as follows. The language of $\mathcal{M}$ consists of a partial order $\leq$ and infinitely many unary relations $P_e$, $e\in\omega$. The relations $P^{\mathcal{M}}_e$, $e\in\omega$, are pairwise disjoint. If $a\leq b$ inside $\mathcal{M}$, then $a$ and $b$ must satisfy the same $P_e$. The $\leq$-substructure $P_e(M)$ contains (copies of) $\mathcal{A}_e$ and $\mathcal{B}_e$, in a disjoint way.
\setcounter{theorem}{\value{theorem}-1}
\begin{claim}
	The structure $\mathcal{M}$ is b.e. trivial.
\end{claim}
\begin{proof}
	Let $\widehat{\mathcal{M}}$ be a bi-embeddable copy of $\mathcal{M}$. Since $\widehat{\mathcal{M}}\hookrightarrow \mathcal{M}$, the relations $P_e^{\widehat{\mathcal{M}}}$ are pairwise disjoint, and the $\leq$-substructure $P_e(\widehat{M})$ consists of two disjoint well-orders $\widehat{\mathcal{A}}_e$ and $\widehat{\mathcal{B}}_e$. Since $\widehat{\mathcal{M}} \approx \mathcal{M}$ and every well-order is b.e. trivial, we deduce that the posets $P_e(\widehat{M})$ and $P_e(M)$ are isomorphic. Therefore, we have $\widehat{\mathcal{M}} \cong \mathcal{M}$.
\end{proof}

We show that the b.e. categoricity spectrum of the structure $\mathcal{M}$ coincides with the set of all $PA$ degrees over $\mathbf{0}^{(2\beta + 2)}$.

For a natural number $e$, let $a_e$ and $b_e$ be the $\leq_{\mathcal{M}}$-least elements inside $\mathcal{A}_e$ and $\mathcal{B}_e$, respectively. W.l.o.g., we can assume that given $e$, one can effectively compute $a_e$ and $b_e$.

\begin{claim}\label{lem:upper-bound}
	Let $\mathcal{N}$ be a computable structure isomorphic to $\mathcal{M}$, and let $\mathbf{d}$ be a $PA$ degree over $\mathbf{0}^{(2\beta+2)}$. There is a $\mathbf{d}$-computable isomorphism from $\mathcal{M}$ onto $\mathcal{N}$.
\end{claim}
\setcounter{theorem}{\value{theorem}+1}
\begin{proof}
	Given $\mathcal{N}$, it is not hard to produce two computable sequences $(c_e)_{e\in\omega}$ and $(d_{e})_{e\in\omega}$ of elements from $\mathcal{N}$ such that for every $e$:
	\begin{enumerate}
		\item both $c_e$ and $d_e$ satisfy $P_e$ inside $\mathcal{N}$,

		\item $c_e$ and $d_e$ are $\leq_{\mathcal{N}}$-incomparable, and

		\item $c_e <_{\mathbb{N}} d_e$, where $\leq_{\mathbb{N}}$ is the standard ordering of natural numbers.
	\end{enumerate}
	By $\mathcal{C}_e$ we denote the $\leq_{\mathcal{N}}$-substructure of $\mathcal{N}$ containing all elements comparable with $c_e$. Similarly, the linear order $\mathcal{D}_e$ consists of the elements comparable with $d_e$. Clearly, the computable indices of $\mathcal{C}_e$ and $\mathcal{D}_e$ can be recovered effectively in $e$.

	Consider a computable $\Sigma_{2\beta+2}$ sentence
  	\[
		\xi = \exists x \forall y ( x \leq y \rightarrow y\sim_{\beta} x).
	\]
	It is not hard to show that the well-order $\omega^{\beta+1}$ does not satisfy $\xi$. On the other hand, the ordinal $(\omega^{\beta+1} + \omega^{\beta})$ satisfies $\xi$ (just choose any $x$ from the $\omega^{\beta}$-part of the ordinal).

	Let $U:= \{ e\,\colon \mathcal{C}_e \models \xi\}$ and $V:= \{e \,\colon \mathcal{D}_e \models \xi\}$. It is not difficult to show that $U$ and $V$ have the following properties:
	\begin{itemize}
		\item[(a)] $U$ and $V$ are disjoint $\Sigma^0_{2\beta+2}$ sets.

		\item[(b)] If $e\in U$, then $\mathcal{C}_e \cong \omega^{\beta+1} + \omega^{\beta}$ and $\mathcal{D}_e \cong \omega^{\beta+1}$.

		\item[(c)] If $e\in V$, then $\mathcal{C}_e \cong \omega^{\beta+1}$ and $\mathcal{D}_e \cong \omega^{\beta+1}  + \omega^{\beta}$.

		\item[(d)] If $e\not\in U \cup V$, then $\mathcal{C}_e \cong \mathcal{D}_e \cong\omega^{\beta+1}$.
	\end{itemize}
	Fix strongly $\mathbf{0}^{(2\beta+2)}$-computable sequences of finite sets $(U^s)_{s\in\omega}$ and $(V^s)_{s\in\omega}$ such that for each $W \in \{ U, V\}$, we have $W = \bigcup_{s\in\omega} W^s$ and $W^s \subseteq W^{s+1}$ for all $s$.

	We define a $\mathbf{0}^{(2\beta+2)}$-computable tree $T\subset 2^{<\omega}$ as follows. Suppose that a string $\sigma\in 2^{<\omega}$ has length $s$. Then $\sigma \in T$ if and only if for each $e<s$, the following conditions hold:
	\begin{enumerate}
		\item If $e\in U^s$, then $\varphi_e^{\emptyset^{(2\beta+2)}}(e) \downarrow\ = \sigma(e)$.

		\item If $e\in V^s$, then $\varphi_e^{\emptyset^{(2\beta+2)}}(e) \downarrow\ = 1-\sigma(e)$.
	\end{enumerate}
	Note that the tree $T$ is well-defined: If, say, $e\in U^s$, then $\mathcal{C}_e$ is a copy of $\omega^{\beta+1} + \omega^{\beta}$. Since $\mathcal{N}\cong \mathcal{M}$, we deduce that $\varphi_e^{\emptyset^{(2\beta+2)}}(e) \downarrow\ \in \{ 0, 1\}$.

	Recall that $\mathbf{d}$ is a $PA$ degree over $\mathbf{0}^{(2\beta+2)}$. It is easy to show that the branching function $b_T$ is $\mathbf{0}^{(2\beta+2)}$-computable. Hence, there is a $\mathbf{d}$-computable path $P$ through $T$. 	An easy analysis of the definition of $T$ shows the following: for any $e\in\omega$,
	\begin{enumerate}
		\item if $P(e) = 1$ , then $\mathcal{C}_e \cong \mathcal{A}_e$ and $\mathcal{D}_e \cong \mathcal{B}_e$;

		\item if $P(e) = 0$, then $\mathcal{C}_e \cong \mathcal{B}_e$ and $\mathcal{D}_e \cong \mathcal{A}_e$.
	\end{enumerate}
	Thus, there is a $\mathbf{d}$-computable function $f(e,i)$ with the following property: If $i$ is a computable index of a structure $\mathcal{L}_e \in \{ \mathcal{A}_e, \mathcal{B}_e\}$, then $f(e,i)$ is a computable index of a structure $\mathcal{R}_e \in \{ \mathcal{C}_e, \mathcal{D}_e\}$ such that $\mathcal{R}_e$ is isomorphic to $\mathcal{L}_e$.

	We apply \cref{lem:switching} to the indices provided by the function $f(e,i)$ and recover (uniformly in $e$) a $\Delta^0_{2\beta+2}$-index for a $\mathbf{0}^{(2\beta+2)}$-computable isomorphism $g_e$ from $\mathcal{L}_e$ onto $\mathcal{R}_e$. Since $\mathbf{d} > \mathbf{0}^{(2\beta + 2)}$, one can easily construct a $\mathbf{d}$-computable isomorphism $h \colon \mathcal{M} \cong\mathcal{N}$, extending the isomorphisms $f_e$.
\end{proof}

\cref{lem:upper-bound} implies that every $PA$ degree over $\mathbf{0}^{(2\beta+2)}$ belongs to the b.e.\ categoricity spectrum of $\mathcal{M}$.

Now we define a new computable copy $\mathcal{R}$ of the structure $\mathcal{M}$. We build two computable sequences of linear orders $(\mathcal{C}_e)_{e\in\omega}$ and $(\mathcal{D}_e)_{e\in\omega}$ such that:
\[
	\mathcal{C}_{e} \cong
	\begin{cases}
		\omega^{\beta+1} + \omega^{\beta} & \text{if } \varphi^{\emptyset^{(2\beta + 2)}}_e(e)\downarrow\  \in \{0, 1\},\\
		\omega^{\beta+1} & \text{otherwise};
	\end{cases}
	\quad\text{and}\quad
	\mathcal{D}_e \cong \omega^{\beta+1}.
\]
For every $e\in\omega$, the $P_e$-part of the structure $\mathcal{R}$ contains copies of $\mathcal{C}_e$ and $\mathcal{D}_e$, in a disjoint way. Clearly, $\mathcal{R}$ is isomorphic to $\mathcal{M}$.

Let $c_e$ and $d_e$ be the least elements inside $\mathcal{C}_e$ and $\mathcal{D}_e$, respectively. As before, we assume that one can compute $c_e$ and $d_e$, uniformly in $e$.

Suppose that $f$ is an arbitrary isomorphic embedding from $\mathcal{M}$ into $\mathcal{R}$. We define a $\deg_T(f)$-computable set $X$ as follows: a number $e$ belongs to $X$ if and only if the element $f(a_e)$ is comparable with $c_e$ inside $\mathcal{R}$.

Note that the ordinal $\omega^{\beta+1} + \omega^{\beta}$ cannot be isomorphically embedded into the well-order $\omega^{\beta+1}$. Therefore, the set $X$ has the following properties:
\begin{itemize}
	\item[(a)] If $\varphi^{\emptyset^{(2\beta+2)}}_e(e) \downarrow\ = 1$, then $\mathcal{A}_e\cong \mathcal{C}_e\cong \omega^{\beta+1} + \omega^{\beta}$ and hence, $e\in X$.

	\item[(b)] If $\varphi^{\emptyset^{(2\beta+2)}}_e(e) \downarrow\ = 0$, then $\mathcal{B}_e\cong \mathcal{C}_e\cong \omega^{\beta+1} + \omega^{\beta}$ and $e\not\in X$.
\end{itemize}

By Proposition~\ref{prop:PA-char}, we obtain that $\deg_T(f)$ is a $PA$ degree over $\mathbf{0}^{(2\beta + 2)}$. This shows that every degree from the b.e. categoricity spectrum of $\mathcal{M}$ is a $PA$ degree over $\mathbf{0}^{(2\beta + 2)}$.

The proof for the case $\alpha = 2\beta + 1$ is very similar, modulo the following key point: one needs to employ the well-orders $\omega^{\beta}$ and $\omega^{\beta}\cdot 2$ in place of $\omega^{\beta+1}$ and $\omega^{\beta+1} + \omega^{\beta}$, respectively. The proof for finite $\alpha$ (either even or odd) can be obtained via minor modifications.
Theorem~\ref{theo:PA-deg} is proved.
\end{proof}

\begin{cor}
	There is a structure $\mathcal S$ without degree of bi-embeddable categoricity but $\catspb(\mathcal S)\cap \mathrm{HYP}\neq \emptyset$ and $\catspb(\mathcal S)$ does not have minimal elements.
\end{cor}


\section{Categoricity vs bi-embeddable categoricity spectra}
In this section, we explore the connections between categoricity and b.e.\ categoricity spectra. We begin by discussing limitations to b.e.\ categoricity spectra, by showing examples of sets of degrees that are not b.e. categoricity spectra, and degrees that are not degrees of b.e.\ categoricity. First, we give a necessary condition for a degree to be a degree of bi-embeddable categoricity and show that any b.e.\ categoricity spectrum is meager.  It is still open whether there are examples separating categoricity spectra from b.e. categoricity spectra.

\begin{theorem}\label{theo:degreesarehyp}
	Every degree of bi-embeddable categoricity is hyperarithmetical.
\end{theorem}
\begin{proof}
  Let $\mathbf{d}\not\in \mathrm{HYP}$, and let $\mathcal A$ be a computable structure. We show that $\mathbf d$ is not a degree of b.e.\ categoricity for $\mathcal A$. Towards this let $\mathcal A_0,\mathcal A_1,\dots$ be an enumeration of all computable structures bi-embeddable with $\A$. For $f,g:\omega\ra \omega$ let $(f,g):\omega\ra \omega$ be defined as $(f,g)(2x)=f(x)$ and $(f,g)(2x+1)=g(x)$. Then the set
  \[ \{(f,g): f:\A_0\hookrightarrow \A_1\mathbin{\&} g:\A_1 \hookrightarrow\A_0\}\]
  is $\Pi^0_2$, and thus $\Sigma_1^1$. Therefore, by Kreisel's Basis Theorem~\cite[Theorem 7.2]{sacks1990}, there exists a pair of embeddings $(f_1:\A_0\hookrightarrow \A_1,g_1:\A_1\hookrightarrow\A_0)$ such that $\mathbf d\not \leq_h f_1$. Suppose we are given pairs of embeddings $(f_i:\A_0\hookrightarrow \A_i,g_i:\A_i\hookrightarrow \A_0)$ for $1\leq i\leq n$  with $\mathbf d \not\leq_h\bigoplus_{1\leq i\leq n} (f_i,g_i)$. Then, by Kreisel's Basis Theorem relativized to $ \bigoplus_{1\leq i\leq n} (f_i,g_i)$, there exists a pair of embeddings $(f_{n+1},g_{n+1})$ such that $\mathbf d\not \leq_h \bigoplus_{1\leq i\leq n+1} (f_i,g_i)$. Now, let $\mathbf a,\mathbf b$ be an exact pair for the sequence $(\bigoplus_{1\leq j\leq i} (f_j,g_j))_{i\in\omega}$. Then $\mathbf a$ and $\mathbf b$ can compute a pair of embeddings between any two bi-embeddable copies of $\A_0$. This implies, that if $\mathbf d$ is a degree of categoricity for $\A$, then $\mathbf d\leq \mathbf a$ and $\mathbf d\leq \mathbf b$. However, since $\mathbf a$ and $\mathbf b$ are an exact pair, this implies that $\mathbf d \leq_T \bigoplus_{1\leq i\leq n} (f_i,g_i)$ for some $n$, a contradiction.
\end{proof}

We write $\A\approx_{\mathbf{d}}\B$ if there are two $\mathbf{d}$-computable embeddings witnessing that $\A\approx\B$.

\begin{theorem}
	For every computable structure $\mathcal S$, $\catspb(\mathcal S)$ either coincides with all Turing degrees or  is meager.
\end{theorem}

\begin{proof}
Assume that $\mathbf{0}\notin\catspb(\mathcal S)$ and let $\A \approx\mathcal{S}$ such that $\mathbf{0}\notin\{\mathbf{d}: \A \approx_{\mathbf{d}} \mathcal{S}\} $. We will show that $\{\mathbf{d}: \A \approx_{\mathbf{d}} \mathcal{S}\} $ is meager. The theorem will follow by observing that
\[
\catspb(\mathcal S)=\bigcap_{\A \in [\mathcal{S}]_\approx} \{\mathbf{d}: \A \approx_{\mathbf{d}} \mathcal{S}\}.
\]
First, note that $\{\mathbf{d}: \A \approx_{\mathbf{d}} \mathcal{S}\}= \bigcup_{\langle e,i\rangle \in\omega} P_{\langle e,i\rangle}$, where
\[
P_{\langle e,i\rangle}=\{X: \phi^X_e: \A \hookrightarrow \mathcal{S} \mbox{  and } \phi^X_i: \mathcal{S} \hookrightarrow \mathcal{A}\}.
\]
We will prove that all $P_{\langle e,i\rangle}$'s are nowhere dense. Given $\sigma\in 2^{\omega}$, look for $\tau\supseteq \sigma$ such that, for all $X\supseteq \tau$,  one of the three following holds: $\phi^X_e$ or $\phi^X_i$ is nontotal; $\phi^X_e$ is not embedding from $\A$ to $\mathcal{S}$; $\phi^X_i$ is not an embedding from $\mathcal{S}$ to $\mathcal{A}$. We distinguish two cases.

\begin{enumerate}
\item[\emph{(i)}] If such $\tau$ exists, then $\{X: X\supseteq \tau\}\subseteq \overline{P_{\langle e,i\rangle}}$.
\item[\emph{(ii)}] If there is no such $\tau$, we claim that $\A$ and $\mathcal{S}$ are computably bi-embeddable, a contradiction. To see this,   given $\sigma$ define two computable sets $Y_0=\cup_{k\in\omega} \rho_k$ and $Y_1=\cup_{k\in\omega} \xi_k$, where
\[
\rho_0=\sigma, \quad \rho_{k+1}=\mbox{ the least  $\rho\supseteq \rho_k$ such that  $\phi_e^{\rho}(k)\downarrow$};
\]
\[
\xi_0=\sigma, \quad \xi_{k+1}=\mbox{ the least  $\xi\supseteq \xi_k$ such that  $\phi_i^{\xi}(k)\downarrow$}.
\]
Since \emph{(i)} does not hold, it must be the case that both $\phi_e^{Y_0}$ and $\phi_i^{Y_1}$ are total, and moreover $\phi^{Y_0}_e$ is an embedding from $\A$ to $\mathcal{S}$ and $\phi^{Y_1}_i$ is an embedding from $\mathcal{S}$ to $\A$. So, this case never holds.
%
%
%
\end{enumerate}
Thus, $P_{\langle e,i\rangle}$ is nowhere dense, giving that  $\{\mathbf{d}: \A \approx_{\mathbf{d}} \mathcal{S}\} $ is meagre.
\end{proof}

By calculating the complexity of the forcing condition in the above proof, it is not hard to show that there is a $\Delta^0_3$ degree that can not be a degree of b.e.\ categoricity. The same result holds for degree of categoricity. Anderson and Csima~\cite{anderson2012} also proved that no noncomputable hyperimmune-free degree can be a degree of categoricity. Their proof extends with almost no modification to b.e.\ categoricity.

The study of which degrees can not be degrees of categoricity recently motivated the notion of  \emph{lowness for isomorphism}~\cite{FS-14}. A Turing degree $\mathbf{d}$ is \emph{low for isomorphism} if for any computable structures $\A\cong\B$, the existence of a $\mathbf{d}$-computable isomorphism from $\A$ onto $\B$ implies that $\A$ and $\B$ are already computably isomorphic.

The next definition gives two variants of how one can formalize a notion of lowness in the setting of isomorphic embeddings. Proposition~\ref{prop:equivalent-low} shows that the two variants turn to be equivalent.



\begin{definition}
	Let $\mathbf{d}$ be a Turing degree. The degree $\mathbf{d}$ is \emph{low for embeddings} if for any computable structures $\mathcal{A}$ and $\mathcal{B}$, we have
	\[
		(\exists f \leq_T \mathbf{d}) (f \colon \mathcal{A} \hookrightarrow \mathcal{B}) \ \Rightarrow\ (\exists g \equiv_T \mathbf{0}) (g \colon \mathcal{A} \hookrightarrow \mathcal{B}).
	\]

	The degree $\mathbf{d}$ is \emph{low for bi-embeddings} if for any computable structures $\mathcal{A}$ and $\mathcal{B}$,
	\[
		(\mathcal{A} \approx_{\mathbf{d}} \mathcal{B}) \ \Rightarrow (\mathcal{A} \approx_{\mathbf{0}} \mathcal{B}).
	\]
\end{definition}

\begin{proposition}\label{prop:equivalent-low}
	A degree $\mathbf{d}$ is low for embeddings if and only if it is low for bi-embeddings.
\end{proposition}
\begin{proof}
	It is clear that every low for embeddings degree is also low for bi-embeddings. Suppose that a degree $\mathbf{d}$ is not low for embeddings, i.e. there exist computable structures $\mathcal{A}$ and $\mathcal{B}$ such that there is a $\mathbf{d}$-computable isomorphic embedding $f\colon  \mathcal{A} \hookrightarrow \mathcal{B}$, but $\mathcal{A}$ is not computably embeddable into $\mathcal{B}$. W.l.o.g., one may assume that the domains of $\mathcal{A}$ and $\mathcal{B}$ are both equal to $\omega$.

	We define new computable structures $\mathcal{A}^\ast$ and $\mathcal{B}^\ast$ as follows.
	\begin{itemize}
		\item The language of our structures contains the language of $\mathcal{A}$ and a new equivalence relation $E$.

		\item The structure $\mathcal{B}^{\ast}$ is a disjoint union of infinitely many copies of $\mathcal{B}$. For $i\in\omega$, the $i^{\text{th}}$ copy of $\mathcal{B}$ inside $\mathcal{B}^{\ast}$ has domain $\{ \langle i,x\rangle \,\colon x\in\omega\}$. Each copy of $\mathcal{B}$ forms an $E$-equivalence class.

		\item The structure $\mathcal{A}^{\ast}$ is arranged similarly to $\mathcal{B}^{\ast}$, modulo the following modification: The $0^{\text{th}}$ copy of $\mathcal{B}$ should be replaced by a copy of $\mathcal{A}$.
	\end{itemize}

	It is not hard to show that $\mathcal{A}^{\ast}$ and $\mathcal{B}^{\ast}$ have the following properties:
	\begin{enumerate}
		\item $\mathcal{A}^{\ast}$ and $\mathcal{B}^{\ast}$ are bi-embeddable.

		\item There is a computable embedding from $\mathcal{B}^{\ast}$ into $\mathcal{A}^{\ast}$: just map the $i^{\text{th}}$ copy of $\mathcal{B}$ inside $\mathcal{B}^{\ast}$ onto the $(i+1)^{\text{th}}$ copy of $\mathcal{B}$ inside $\mathcal{A}^{\ast}$.

		\item There is a $\mathbf{d}$-computable embedding from $\mathcal{A}^{\ast}$ into $\mathcal{B}^{\ast}$: The embedding uses $f$ to map the copy of $\mathcal{A}$ inside $\mathcal{A}^{\ast}$ into the $0^{\text{th}}$ copy of $\mathcal{B}$ inside $\mathcal{B}^{\ast}$. All the other $E^{\mathcal{A}}$-classes are mapped in a straightforward way.
	\end{enumerate}
	Therefore, we obtain that $\mathcal{A} \approx_{\mathbf{d}} \mathcal{B}$. On the other hand, if $h\colon \mathcal{A}^{\ast} \hookrightarrow \mathcal{B}^{\ast}$, then the function
	\[
		h^{\#} := h\upharpoonright \{ \langle 0,x\rangle \,\colon x\in\omega\}
	\]
	induces an isomorphic embedding from $\mathcal{A}$ into $\mathcal{B}$ (or more formally, $h^{\#}$ is an embedding from the copy of $\mathcal{A}$ inside $\mathcal{A}^{\ast}$ into the $i^{\text{th}}$ copy of $\mathcal{B}$ inside $\mathcal{B}^{\ast}$, for some $i$). Moreover, $h^{\#}\leq_T h$. Thus, there is no computable embedding from $\mathcal{A}^{\ast}$ into $\mathcal{B}^{\ast}$, and $\mathcal{A}^{\ast} \not\approx_{\Delta^0_1} \mathcal{B}^{\ast}$. In other words, the structures $\mathcal{A}^{\ast}$ and $\mathcal{B}^{\ast}$ witness that $\mathbf{d}$ is not low for bi-embeddings.
\end{proof}

%
%
%
%
%
%
%
\subsection{Case-study: structures that are not b.e.\ trivial}
So far in all of our results we built structures that are bi-embeddably trivial, i.e., their isomorphism type and their bi-embeddability type coincide. Furthermore, it is not hard to see that in all of these examples the degree of categoricity and the degree of bi-embeddable categoricity coincide. But what can we say about structures which are not b.e. trivial? One example that comes to mind is $\eta$, the ordering of the rational numbers. It is well known that $\eta$ is computably categorical. However, it is not hard to see that $\eta$ is not hyperarithmetically b.e.\ categorical and therefore does not have a degree of bi-embeddable categoricity.
\begin{proposition}\label{prop:etanonhyp}
	There is a computably categorical linear ordering which is not hyperarithmetically bi-embeddably categorical.
\end{proposition}
\begin{proof}
  That $\eta$ is computably categorical follows easily by a basic back-and-forth argument. Furthermore, it is easy to see that $\eta$ is bi-embeddable with every countable linear ordering which has a dense subordering. Therefore it is bi-embeddable with the Harrison linear ordering $\omega_1^{\mrm{CK}}\cdot(1+\eta)$. Take any embedding from a standard copy of $\eta$ to a computable copy of the Harrison linear ordering that does not have hyperarithmetic descending sequences. Then this embedding computes a descending sequence and thus it can not be hyperarithmetic.
\end{proof}
\begin{proposition}
	 For every computable ordinal $\beta$ there is a proper $\Delta^0_{\beta}$ computably categorical linear ordering which is not hyperarithmetically bi-embeddably categorical.
\end{proposition}
\begin{proof}
	(1) First, assume that the ordinal $\beta$ is even, i.e. $\beta = 2\alpha$. Notice that $\omega^\alpha+1$ has a formally $\Sigma^0_{2\alpha}$ Scott family (see \cref{lem:oddordinalsscottfam}). It furthermore follows from results of Ash~\cite{ash1986} that it does not have a Scott family of less complexity. It is now easy to see that $\omega^\alpha+1+\eta$ has a $\Sigma^0_{2\alpha}$ Scott family with one parameter. Therefore it is relatively $\Delta^0_{2\alpha}$ categorical. However, $\omega^\alpha+1+\eta$ is bi-embeddable with $\eta$ which by \cref{prop:etanonhyp} is not hyperarithmetically bi-embeddably categorical.

	(2) Suppose that $\beta = 2\alpha + 1$, where $\alpha$ is non-zero. We prove that the order $\mathcal{L} = \omega^{\alpha} \cdot (1 + \eta)$ has the desired properties.

	We sketch the description of a formally $\Sigma^0_{2\alpha+1}$ Scott family for $\mathcal{L}$. A typical example of a Scott formula $\psi(\bar x)$ is constructed as follows. Consider a tuple $\bar a = a_0, a_1, a_2, a_3$ from $\mathcal{L}$ such that:
	\begin{itemize}
		\item $a_0$ is the least element of $\mathcal{L}$.

		\item $[a_0;a_1)_{\mathcal{L}} \cong \omega^{\gamma}$, where $\gamma < \alpha$.

		\item $a_2$ and $a_3$ belong to the same copy of $\omega^{\alpha}$ inside $\mathcal{L}$, and this copy is \emph{not} the first. Let $c$ be the least element in the copy.

		\item $[c;a_2)_{\mathcal{L}} \cong \omega^{\delta_0}$ and $[a_2;a_3)_{\mathcal{L}} \cong \omega^{\delta_1}$, where $\delta_0,\delta_1<\alpha$.
	\end{itemize}
	Then the Scott formula $\psi(x_0,x_1,x_2,x_3)$ such that $\mathcal{L}\models \psi(\bar a)$ is defined as a conjunction of the following formulas:
	\begin{itemize}
		\item $x_0 < x_1 < x_2 < x_3$;

		\item $\forall y (x_0 \leq y)$;

		\item a computable $\Pi_{2\gamma+1}$ formula saying that the interval $[x_0;x_1)$ is isomorphic to $\omega^{\gamma}$;

		\item a computable $\Sigma_{2\alpha+1}$ formula postulating the following: there is an element $c$ such that
		\begin{itemize}
			\item $c$ is not the least,

			\item $\forall z<c(z \not\sim_{\alpha} c)$,

			\item $[c;x_2)\cong \omega^{\delta_0}$;
		\end{itemize}

		\item a $\Pi^c_{2\delta_1+1}$-formula saying that  $[x_2;x_3)\cong \omega^{\delta_1}$.
	\end{itemize}
	It is not hard to show that the formula $\psi$ is equivalent to a computable $\Sigma_{2\alpha+1}$ formula. Furthermore, the description of $\psi$ can be easily extended to a construction of a $\Sigma^0_{2\alpha+1}$ Scott family for $\mathcal{L}$.

	The results of Ash and Knight (see Theorem~4.2 and p.~224 in~\cite{ash1990a}) imply that for any $\Sigma^0_{2\alpha}$ set $X$, one can build a uniformly computable sequence $(\mathcal{C}_n)_{n\in\omega}$ of linear orders such that
	\[
		\mathcal{C}_n \cong
		\begin{cases}
			\text{an ordinal } \gamma < \omega^{\alpha} & \text{if } n\in X,\\
			 \omega^{\alpha} \cdot (1 + \eta) & \text{if } n\not\in X.
		\end{cases}
	\]
	The formal details of this construction can be recovered, e.g., from a similar proof of Theorem~3.1 in \cite{Bazh16-ba}. Take $X$ as a $\Sigma^0_{2\alpha}$ complete set, and consider the order
	\[
		\mathcal{M} := \sum_{n\in\omega} (1 + \mathcal{C}_n).
	\]
	Clearly, $\mathcal{M}$ is a computable copy of $\omega^{\alpha}\cdot (1+\eta)$. Moreover, the $\alpha$-block relation $\sim_{\alpha}$ inside $\mathcal{M}$ is a $\Sigma^0_{2\alpha}$ complete set.

	On the other hand, it is easy to build a computable copy $\mathcal{N}$ of $\omega^{\alpha}\cdot (1+\eta)$ such that the relation $\sim_{\alpha}$ for $\mathcal{N}$ is computable. Thus, every isomorphism from $\mathcal{M}$ onto $\mathcal{N}$ must compute a $\Sigma^0_{2\alpha}$ complete set. This implies that the order $\mathcal{L} = \omega^{\alpha}\cdot (1 + \eta)$ is not $\Delta^0_{2\alpha}$ categorical and thus an example of a properly $\Delta^0_{2\alpha+1}$ categorical linear ordering.

	In order to finish the proof, notice that the order $\mathcal{L}$ is bi-embeddable with $\eta$. In turn,  $\eta$ is not hyperarithmetically bi-embeddably categorical.
\end{proof}

\section{Index sets}
In this section we prove results on the complexity of index sets of $\mathbf 0'$ bi-embeddably categorical structures and the index set of structures with degree of bi-embeddable categoricity. The structures we will construct in our proofs belong to the class of strongly locally finite graphs. Recall that a graph is \emph{strongly locally finite} if all of its connected components are finite. It is easy to see that computable strongly locally finite graphs have formally $\Sigma^0_2$ Scott families and are thus $\mathbf 0'$ computably categorical. The following result about strongly locally finite graphs will be used in the following proofs.
\begin{proposition}\label{prop:nonhypslfg}
  There is a strongly locally finite graph that is not hyperarithmetically bi-embeddably categorical.
\end{proposition}
\begin{proof}
Let $H\subseteq \omega^{<\omega}$ be a computable tree without hyperarithmetic paths. We build a strongly locally finite graph $G_H$ such that the partial ordering under embeddability of its components is computably isomorphic to $H$.

For any $\sigma\in H$, $G_H$ contains the component $C_\sigma$: A ray of length $|\sigma|+1$ where the first vertex has a loop connected to it and the $(i+2)^{th}$ vertex for $i< |\sigma|$ has a cycle of length $\sigma(i)+2$ attached. Clearly the partial ordering of the components is computably isomorphic to $H$ by $C_\sigma \mapsto \sigma$. Now $G_H$ has a bi-embeddable copy $\tilde G$ that skips a fixed $C_\sigma$ such that $\sigma$ lies on a path in $H$. Now consider embeddings $\mu: G_H \rightarrow G$ and $\nu: G\rightarrow G_H$, then $C_\sigma \subset \mu(C_\sigma) \subset \nu(\mu(C_\sigma))\subset \dots$ and thus there is $f\in [H]$ hyperarithmetic in $\mu\oplus \nu $. Hence, $\mu \oplus \nu$ itself can not be hyperarithmetic.
\end{proof}
\begin{theorem}\label{theo:indexset0'}
The index set of $\mathbf{0}'$-computably bi-embeddably categorical structures is $\Pi^1_1$-complete.
\end{theorem}
\begin{proof}
Let $H$ be a computable tree without hyperarithmetic paths as in the proof of \cref{prop:nonhypslfg} and let $(T_i)_{i\in \omega}$ be a uniformly computable sequence of trees such that $T_i$ is well-founded iff $i\in \mathcal O$. For two strings $\sigma$, $\tau$ of the same length let $\sigma\star\tau=\sigma_0\tau_0\sigma_1\tau_1\dots \sigma_{|\sigma|-1}\tau_{|\tau|-1}$, and consider the sequence of trees $(S_i)_{i\in \omega}$
\[ S_i=\{ \xi : \xi \subseteq \sigma\star\tau,\ |\sigma|=|\tau|,\ \sigma\in T_i,\ \tau\in H\}.\]
Clearly, it is uniformly computable, and $S_i$ is well-founded iff $i\in \mathcal O$. Furthermore, no path in $[S_i]$ is hyperarithmetic. Using the same coding as in the proof of \cref{prop:nonhypslfg} we get that if $i\in \mathcal{O}$, then $G_{S_i}$ is b.e.\ trivial and thus $\mathbf 0'$-computably bi-embeddably categorical. If $i\not \in \mathcal O$, then $G_{S_i}$ is not $\mathbf 0^{(\alpha)}$-computably bi-embeddably categorical for $\alpha<\omega_1^{\mathrm{CK}}$.

\end{proof}

Note that in~\cite{downey2015a}, it was shown that the index set of computably categorical structures is $\Pi^1_1$-complete. We leave open whether a similar result can be obtained for computably bi-embeddably categorical structures.

\begin{cor}
	The index set of structures with degree of bi-embeddable categoricity is $\Pi^1_1$ complete.
\end{cor}
\begin{proof}
  In the proof of \cref{theo:indexset0'} we produced a uniformly computable sequence of structures $(G_{S_i})_{i\in\omega}$ such that in the $\Pi_1^1$ outcome $G_{S_i}$ is $\mathbf 0'$-computably bi-embeddably categorical. To obtain the corollary we take the cardinal sum of $G_{S_i}$ and a structure which has degree of bi-embeddable categoricity $\mathbf 0'$. More formally, the new structure $\mathcal S_i$ is in the language of graphs with an additional relation symbol $R/1$ such that $R$ partitions $\omega$ into two infinite sets. We let $R(S_i)\cong G_{S_i}$, and we let the corelation of $R$ be isomorphic to the canonical unbounded equivalence structure -- the equivalence structure $\mathcal E$ having one equivalence class of each size -- i.e., the universe of $\mathcal E$ is $\{ \langle i, n\rangle: n<i\}$ and its equivalence relation (edge relation) is defined by $\langle i_1,n_1\rangle E\langle i_1,n_2\rangle\LR i_1=i_2$.

  The structure $\mathcal E$ has degree of b.e.\ categoricity $\mathbf 0'$~\cite[Theorem 3.8]{bazhenov2018a}. Now, every embedding of $\mathcal S_i$ into a bi-embeddable copy computes an embedding between $\neg R(S_i)$ and a bi-embeddable copy of $\mathcal E$. If $i\in \mathcal O$, then between any two bi-embeddable copies of $\mathcal S_i$ there are $\mathbf 0'$ computable embeddings and there are bi-embeddable copies $\A$ and $\B$ such that $\mathbf 0'$ is the least degree computing such embeddings. Thus $\mathcal S_i$ has degree of b.e. categoricity $\mathbf 0'$. On the other hand, if $i\not\in \mathcal O$, then $R(S_i)$ is not hyperarithmetically bi-embeddably categorical, and has by \cref{theo:degreesarehyp} no degree of b.e. categoricity.
\end{proof}
\printbibliography
\end{document}